\documentclass[11pt]{amsart}
\usepackage{amsmath,amsthm, amscd, amssymb, amsfonts, mathrsfs}
\usepackage[all]{xy}
\usepackage{color}
\newcommand\Tstrut{\rule{0pt}{2.6ex}}       % top strut
\newcommand\Bstrut{\rule[-1.1ex]{0pt}{0pt}}

\usepackage[T1]{fontenc}
\allowdisplaybreaks

\newtheorem{lema}{Lemma}
\newtheorem{prop}[lema]{Proposition}

\newtheorem{Theorem}[lema]{Theorem}

\theoremstyle{definition}

\newtheorem{definition}[lema]{Definition}
\newtheorem{Rem}[lema]{Remark}

\newtheorem{example}[lema]{Example}
\newtheorem{examples}[lema]{Examples}

\usepackage{hhline}

\newcommand{\ba}{\mathbf{a}}

\newcommand{\cA}{\mathcal{A}}

\newcommand{\cJ}{\mathcal{J}}

\newcommand{\cS}{\mathcal{S}}
\newcommand{\cE}{\mathcal{E}}

\newcommand{\Sn}{{\mathbb S}}

\newcommand{\B}{{\mathbb B}}
\newcommand{\A}{{\mathbb A}}
\newcommand{\MM}{{\mathbb M}}

\newcommand{\ot}{{\otimes}}

\newcommand{\ku}{\Bbbk}

\newcommand{\N}{{\mathbb N}}

\newcommand{\F}{{\mathbb F}}

\newcommand{\BV}{{\mathcal B}}

\newcommand{\ydh}{{}^H_H\mathcal{YD}}

\newcommand{\ydg}{{}^{\ku G}_{\ku G}\mathcal{YD}}
\newcommand{\ydgdual}{{}^{\ku^G}_{\ku^G}\mathcal{YD}}

\newcommand{\Ss}{{\mathcal S}}

\newcommand{\Aut}{\operatorname{Aut}}
\newcommand{\Ext}{\operatorname{Ext}}
\newcommand{\Ind}{\operatorname{Ind}}

\newcommand\tr{\operatorname{tr}}
\newcommand\Rep{\operatorname{Rep}}

\newcommand\ad{\operatorname{ad}}

\def\pf{\begin{proof}}
\def\epf{\end{proof}}

\newcommand\supp{\operatorname{Supp}}

\newcommand\cop{{\operatorname{cop}}}

\newcommand\gr{{\operatorname{gr}}}

\newcommand\Irr{\operatorname{Irr}}
\newcommand\Aff{\operatorname{Aff}}

\newcommand\e{\varepsilon}
\newcommand\be{\mathbf{e}}

\newcommand{\bao}[1]{
\settowidth{\ancho}{$#1$}
\settoheight{\alto}{$#1$}
\hbox{\vbox{  \hbox{{${}_\ba$}}\kern \alto}\vbox{\hrule height 0.5pt width \ancho \kern 1mm \hbox{$#1$}}}
}

\begin{document}

\title[Representations of copointed Hopf algebras]{Representations of copointed Hopf algebras arising from the tetrahedron rack}
\author[pogorelsky and vay]{B\'arbara Pogorelsky and Cristian Vay}

\address{Instituto de Matem\'atica, Universidade Federal do Rio Grande do Sul, Av. Bento Goncalves 9500, Porto Alegre, RS, 91509-900, Brazil} \email{barbara.pogorelsky@ufrgs.br}

\address{FaMAF-CIEM (CONICET), Universidad Nacional de C\'ordoba, Medina A\-llen\-de s/n, Ciudad Universitaria, 5000 C\' ordoba, Rep\'ublica Argentina.} \email{vay@famaf.unc.edu.ar}

\thanks{\noindent 2000 \emph{Mathematics Subject Classification.}16W30. \newline B. Pogorelsky was partially supported by Capes-Brazil. C. Vay was partially supported by ANPCyT-Foncyt, CONICET, MinCyT (C\'ordoba)  and Secyt (UNC)}

\maketitle

\smallbreak

\begin{abstract}We study the copointed Hopf algebras attached to the Ni\-chols algebra of the affine rack $\Aff(\F_4,\omega)$, also known as tetrahedron rack,  and the $2$-cocycle $-1$.
We investigate the so-called Verma modules and classify all the simple modules. We conclude that these algebras are of wild representation type and not quasitriangular, also we analyze when these are spherical.

\end{abstract}

\setcounter{tocdepth}{1}

%\tableofcontents

%\tableofcontents \setcounter{section}{0}
\section{Introduction}
\quad We work over an algebraically closed field $\ku$ of characteristic zero. Let $G$ be a finite non-abelian group and let $\ku^G$ denote the algebra of functions on $G$. A Hopf algebra with coradical isomorphic to $\ku^G$ for some $G$ is called {\it copointed}. Nicol\'as Andruskiewitsch and the second author began the study of the copointed Hopf algebras by classifying those finite-dimensional with $G=\Sn_3$ in \cite{AV} and by analyzing the representation theory of them in \cite{AV2}.

Since $\ku^G$ is a commutative semisimple algebra, the representation theory of a copointed Hopf algebra over $\ku^G$ is studied in \cite{AV2} by analogy with the representation theory of semisimple Lie algebras, with $\ku^G$ playing the role of the Cartan subalgebra and the induced modules from the simple one-dimensional $\ku^G$-modules as Verma modules.

\smallbreak

There are few examples of Nichols algebras of finite dimension over non-abelian groups, see for instance \cite{grania,HLV}. In particular, those arising from affine racks are only seven, including the tetrahedron rack. If $X$ is one of these affine racks, then all the liftings of the Nichols algebra $\BV(-1,X)$ over $\ku^G$ were classified in \cite{agustinvay}, where $G$ is any group admitting a principal YD-realization of $X$ with constant $2$-cocycle $-1$. Also the liftings of $\BV(X,-1)$ over the group algebra $\ku G$ were classified in \cite{agustinvay}.

The notation used in the following is explained in Section \ref{sec:pre}. Let $G$ be a finite group and $V\in\ydgdual$ arising from a faithful principal YD-realization of the tetrahedron rack with constant $2$-cocycle $-1$. The Nichols algebra $\BV(V)$ has dimension $72$. The ideal of relations of $\BV(V)$ is generated by four quadratic elements and a single extra one, of degree six, which we denote by $z$. By \cite{agustinvay}, the liftings of $\BV(V)$ over $\ku^G$ are the copointed Hopf algebras $\{\cA_{G,\lambda}\}_{\lambda\in\ku}$, in which the quadratic relations of $\BV(V)$ still hold and the $6$-degree relation $z=0$ deforms to $z=\lambda(1-\chi_z^{-1})\in\ku^G$.

The goal of this paper is to investigate the representation theory of the family $\{\cA_{G,\lambda}\}_{\lambda\in\ku}$ following the strategy of \cite{AV2}. We conclude that there are essentially two kinds of Verma modules. Here is an account of our main results which apply to any group $G$ admitting a faithful principal YD-realization of the tetrahedron rack with constant $2$-cocycle $-1$:

$\bullet$ {\it Let $g\in G$. If the element $z=\lambda(1-\chi_z^{-1})$ annihilates the generator of the Verma modules $M_g$, then $M_g$ inherits a structure of $\BV(V)$-module such that it is a free $\BV(V)$-module of rank $1$, see Lemma \ref{le:delta h en ker}. Hence  $M_g$ has a unique simple quotient of dimension $1$ called $\ku_g$.}

$\bullet$ {\it Otherwise $M_g$ is the direct sum of six $12$-dimensional non isomorphic simple projective modules $L_i^g$, see Lemma \ref{le:dim12}. Tables \ref{table:module L1g}--\ref{table:module L6g} in the Appendix describe the simple modules $L_i^g$.}

$\bullet$ {\it We prove that $\cA_{G,\lambda}$ is of wild representation type, Proposition \ref{le: extensions of one dimensional simple mod}.}

$\bullet$ {\it We give a necessary condition for a copointed Hopf algebra to be quasitriangular, Lemma \ref{le:quasi copointed}. As a consequence $\cA_{G,\lambda}$ is not quasitriangular, Proposition \ref{prop:AGlambda notquasi}.}

$\bullet$ {\it We characterize those $\cA_{G,\lambda}$ which are spherical Hopf algebras, see Proposition \ref{prop:agl spherical}}.

\smallbreak

The other copointed Hopf algebras classified in \cite{agustinvay} are defined by similar relations to $\cA_{G,\lambda}$, roughly speaking a set of quadratic ones and other single relation of bigger degree, but their dimensions are much bigger than $\dim\cA_{G,\lambda}=72|G|$. To extend this work to the other copointed Hopf algebras in \cite{agustinvay}, a better understanding of the corresponding Nichols algebras is needed. We hope that our work will be useful for this purpose.

The paper is organized as follows. In Section \ref{sec:rep of copointed} we analyze the representation theory of copointed Hopf algebras with emphasis in the weight spaces of the modules, we characterize the one-dimensional modules and describe the subalgebra corresponding to the elements of weight $e\in G$. In Section \ref {sec:pre}, we present our main object of study: the algebras $\BV(V)$ and $\cA_{G,\lambda}$. In Section \ref{sec:rep of Aglambda} we concentrate our attention on representations of the algebras $\{\cA_{G,\lambda}\}_{\lambda\in\ku}$. A description of the simple $\cA_{G,\lambda}$-modules is in the Appendix.

\subsection*{Acknowledgments}

\quad The authors thank professor Nicol\'as Andruskiewitsch for proposing this problem and useful suggestions for
this article. The first author also thanks Carolina Renz for her hospitality during her stay in C\'ordoba.

\subsection{Conventions and notation}
\quad We set $\ku^*=\ku\setminus\{0\}$. If $X$ is a set, $\ku X$ denotes the free vector space over $X$.

Let $A$ be a Hopf algebra. Then $\Delta$, $\e$, $\cS$ denote respectively the
comultiplication, the counit and the antipode. The group of group-like elements is $G(A)$. Let ${}^A_A\mathcal{YD}$ be the category of Yetter-Drinfeld modules over $A$. The Nichols algebra $\BV(V)$ of $V\in{}^A_A\mathcal{YD}$ is the graded quotient $T(V)/\cJ(V)$ where  $\cJ(V)$ is the largest Hopf ideal of $T(V)$ generated as an ideal by homogeneous elements of degree $\geq 2$ \cite[2.1]{AS-cambr}.

Let $\{A_{[n]}\}_{n\geq0}$ denote the coradical filtration of $A$. Assume $H=A_{[0]}$ is a Hopf subalgebra. Let $\gr A$ be the graded Hopf algebra associated  to the coradical filtration. Then $\gr A\simeq R\# H$ where $R\in\ydh$ is called the {\it diagram of $A$} and $V=R_{[1]}\in\ydh$ is the {\it infinitesimal braiding} \cite[Definition 1.15]{AS-cambr}. If $R=\BV(V)$, then $A$ is said to be a {\it lifting of $\BV(V)$} ({\it over} $H$).

\smallbreak

Recall that two idempotents $\be,\widetilde{\be}\in A$ are {\it orthogonal} if $\be\widetilde{\be}=0=\widetilde{\be}\be$. An idempotent is {\it primitive} if it is not possible to express it as the sum of two nonzero orthogonal idempotents. A set $\{\be_i\}_{i\in I}$ of idempotents of $A$ is {\it complete} if $1=\sum_{i\in I}e_i$.

Assume $\dim A<\infty$. Then $A$ is a Frobenius algebra, see for instance \cite[Lemma 1.5]{FMS}. Let $\be$ be a primitive idempotent of $A$. Then $top(A\be)=A\be/rad(A\be)$ and the socle $soc(A\be)$ of $A\be$ are simple modules \cite[Theorems 54.11 and 58.12]{CR}. Moreover, $A\be$ is the injective hull of  $soc(A\be)$ and the projective cover of $top(A\be)$ \cite[page 400 and Theorem 58.14]{CR}. We denote by $\Irr A$ a set of representatives of simple $A$-modules.

\section{Representations of copointed Hopf algebras}\label{sec:rep of copointed}

\quad Let $G$ be a finite group, $\ku G$ its group algebra and $\ku^G$ the algebra of functions on $G$. Let $\{g:g\in G\}$ and $\{\delta_g:g\in G\}$ be the dual basis of $\ku G$ and $\ku^G$, respectively. The identity element of $G$ will be denoted by $e$.

If $M$ is a $\ku^G$-module, then $M[g]=\delta_g\cdot M$ is the isotypic component of weight $g\in G$. We denote by $\ku_g$ the one-dimensional $\ku^G$-module of weight $g$. We define
$$M^\times=\oplus_{g\neq e}M[g]\quad\mbox{and}\quad\supp M=\{g\in G:M[g]\neq0\}.$$

\smallbreak

Through this section, {\it $A$ denotes a finite-dimensional copointed Hopf algebra over $\ku^G$, i. ~e. its coradical is isomorphic to $\ku^G$.}

\smallbreak

We consider $A$ as a left $\ku^G$-module via the left adjoint action
$$
\ad\delta_t(a)=\sum_{s\in G}\delta_sa\delta_{t^{-1}s}\quad\forall t\in G,\,a\in A.
$$
By \cite[Lemma 3.1]{AV}, $A=\oplus_{g\in G}A[g]$ is a $G$-graded algebra and
\begin{align}\label{eq:regla de commutacion}
&&&&&&\delta_t a_s=a_s\delta_{s^{-1}t} &&\forall a_s\in A[s], s,t\in G.
\end{align}
If $M$ is an $A$-module, then $M$ is a $\ku^G$-module by restriction. Hence
\begin{align}
&&&&A[g]\cdot M[h]\subseteq M[gh]&&\forall g,h\in G\,\mbox{ by \eqref{eq:regla de commutacion}.}
\end{align}
This means that $M$ is a $G$-graded $A$-module.

We denote by $A_{\ku^G}=A$ as right $\ku^G$-module via the right multiplication. Its isotypic components are $(A_{\ku^G})[g]=A\delta_g$ for all $g\in G$. Note that $A$ is a $\ku^G$-bimodule with the above actions since $\ku^G\subseteq A[e]$.

\smallbreak

Let $R\in\ydgdual$ be the diagram of $A$. Then the multiplication in $A$ induces an isomorphism $R\ot\ku^G\longrightarrow A$ of $\ku^G$-bimodules \cite[Lemma 4.1]{AAnMGV}. Hence we can think of $R$ as a left $\ku^G$-submodule of $A$ and therefore
\begin{align}\label{eq:pesos grales}
&&A[g]=R[g]\,\ku^G\,\mbox{ and }\,(A_{\ku^G})[g]=R\,\delta_g&&\forall g\in G.
\end{align}

Let $g\in G$. As in \cite{AV2}, we define the Verma module of $A$ of weight $g$ as the induced module
$$
M_g=\Ind_{\ku^{G}}^{A}\ku_g=A\ot_{\ku^{G}}\ku\delta_g.
$$
Then $M_g$ is projective, being induced from a module over a  semisimple algebra, and hence injective, because $A$ is Frobenius. By \eqref{eq:regla de commutacion} and \eqref{eq:pesos grales}, the weight spaces satisfy
\begin{align}\label{eq:pesos del verma}
&&M_g[h]=R[hg^{-1}]\delta_g&&\forall h\in G.
\end{align}
Also, $M_g=A\delta_g=R\delta_g$ and $A=\oplus_{g\in G}M_g$.

Notice that if $L$ is a simple $A$-module and $0\neq v\in L[g]$, then $L$ is a quotient of $M_g$ via $\delta_g\mapsto \delta_g\cdot v=v$.

\bigbreak

Let $\be\in A$ be an idempotent. We say that $\be$ is a {\it $g$-idempotent} if $\be\in R[e]\delta_g$. A set $\{\be_i\}_{i\in I}$ of $g$-idempotents is called {\it complete} if $\delta_g=\sum_{i\in I}\be_i$. Next lemma ensures that there always exists a complete set of orthogonal primitive $g$-idempotents.

\begin{lema}\label{lema:g idempotent2} Let $g\in G$, $\mathbf{e}$ be a $g$-idempotent and $\cE_g=\{\be_i\}_{i \in I}$ be a set of orthogonal idempotents of $A$ such that $\delta_g=\sum_{i\in I}\be_i$.
\begin{enumerate}\renewcommand{\theenumi}{\alph{enumi}}\renewcommand{\labelenumi}{
(\theenumi)}
\item\label{item:then g idempotent lema:g idempotent} $\cE_g$ is a complete set of $g$-idempotents.
\smallbreak
\item\label{item:primitive g idempotent lema:g idempotent} $\be$ is primitive if and only if it is not possible to express $\be$ as a sum of orthogonal $g$-idempotents.
\smallbreak
\item\label{item:exists cE lema:g idempotent} There is a complete set of orthogonal primitive $g$-idempotents in $A$.
\smallbreak
\item\label{item:ecdotM lema:g idempotent}$\be\cdot M=\be\cdot M[g]\subseteq M[g]$ for any $A$-module $M$.
\smallbreak
\item\label{item:primitive cota igual lema:g idempotent} If $\#\cE_g=\dim R[e]$, then $\be_i$ is primitive for all $i\in I$. Moreover, if $\be$ is primitive, then $\be=\be_i$ for some $i\in I$.
\item\label{item:primitive cota igual 2 lema:g idempotent} If $\#\cE_g=\dim R[e]$, then $A\be_i\not\simeq A\be_j$ if $i\neq j$.
\end{enumerate}
\end{lema}

\begin{proof}
\eqref{item:then g idempotent lema:g idempotent} We have to prove that $\be_i$ is a $g$-idempotent for all $i\in I$. Fix $i\in I$ and set $\alpha=\be_{i}$ and $\beta=\sum_{i\neq j\in I}\be_j$. If $t\in G$ and $t\neq g$, then $0=\delta_g\delta_t=\alpha\delta_t+\beta\delta_t$. Since $\alpha$ and $\beta$ are orthogonal, $\alpha\delta_t=0$. Hence $\alpha=\alpha\delta_g$ because $1=\sum_{g\in G}\delta_g$. Similarly $\alpha=\delta_g\alpha$. Let $a_{s}\in R[s]$ such that $\alpha=\sum_{s\in G}a_{s}\delta_g$.
Then $\alpha=\delta_g\alpha=\sum_{s\in G}\delta_ga_{s}\delta_g=\sum_{s\in G}a_{s}\delta_{s^{-1}g}\delta_g=a_{e}\delta_g$. That is, $\alpha=\be_{i}$ is a $g$-idempotent.

\eqref{item:primitive g idempotent lema:g idempotent} The first implication is obvious. For the second implication, we proceed as in \eqref{item:then g idempotent lema:g idempotent}. \eqref{item:exists cE lema:g idempotent} follows from \eqref{item:then g idempotent lema:g idempotent} and \eqref{item:primitive g idempotent lema:g idempotent}. \eqref{item:ecdotM lema:g idempotent} holds because $\be\in R[e]\delta_g$.

\eqref{item:primitive cota igual lema:g idempotent} is a consequence of the fact that $\cE_g$ is a basis of $R[e]\delta_g$. Indeed, pick $\alpha=\be_{i}\in\cE_g$ and suppose $\alpha=a+b$ with $a$ and $b$ orthogonal $g$-idempotents of $A$. Then $(Aa)[e]\oplus(Ab)[e]=(A\alpha)[e]=(\ku\cE_g)\alpha=\ku\alpha$ and therefore $a=0$ or $b=0$. For the second statement, we write $\be=\sum_{i\in I}a_i\be_i$ with $a_i\in\ku$, $i\in I$. Since $\be^2=\be$, $a_i=0$ or $1$ for all $i\in I$ and hence $\be=\be_i$ for some $i\in I$.

\eqref{item:primitive cota igual 2 lema:g idempotent} $(A\be_i)[e]=\ku\be_i\neq(A\be_j)[e]=\ku\be_j$ if $i\neq j$. Hence $A\be_i\not\simeq A\be_j$.
\end{proof}

Given a set of idempotents $\cE$ and an $A$-module $M$, we write
$$\supp_\cE M=\{\be\in\cE:\be\cdot M\neq0\}.$$
By \cite[Theorem 54.16]{CR} if $L$ is a simple $A$-module and $\be\in\supp_{\cE}L$, then
$$top(A\be)\simeq L.$$
This allows us to analyze the dimension of the weight spaces of the simple $A$-modules using $g$-idempotents.

\begin{lema}\label{lema:g idempotent} Let $g\in G$ and $\cE_g=\{\be_i\}_{i \in I}$ be a complete set of orthogonal primitive $g$-idempotents. Let $L$ be a simple $A$-module.
\begin{enumerate}\renewcommand{\theenumi}{\alph{enumi}}\renewcommand{\labelenumi}{
(\theenumi)}
\item\label{item:dim weight space lema:g idempotent} $\dim L[g]=\#\supp_{\cE_g}L$.
\smallbreak
\item\label{item:caso part dim weight space lema:g idempotent} If $\#\cE_g=\dim R[e]$ or $1$, then $\dim L[g]=1$ or $0$.
\smallbreak
\item\label{item:particion lema:g idempotent} $\cE_g=\bigcup_{L\in\Irr A}\supp_{\cE_g}L$ is a partition.
\smallbreak
\item\label{item:otra primitive cota lema:g idempotent} $\dim R[e]\geq\sum_{L\in\Irr A}(\dim L[g])^2=\sum_{L\in\Irr A}(\#\supp_{\cE_g}L)^2\geq\#\cE_g$.
\end{enumerate}
\end{lema}

\begin{proof}
\eqref{item:dim weight space lema:g idempotent} By \cite[Theorem 54.16]{CR}, $\dim \be_i\cdot L=1$ for all $\be_i\in\supp_{\cE_g}L$. Pick $w_i\in \be_i\cdot L-\{0\}$ for each $i\in I$. Then $\{w_i:i\in I\}$ is a basis of $L[g]$ since $v=\delta_g\cdot v=\sum_{\be_i\in\supp_{\cE_g}L} \be_i\cdot v$ for all $v\in L[g]$.

\eqref{item:caso part dim weight space lema:g idempotent} If $\#\cE_g=1$, then $\dim L[g]=1$ or $0$ by \eqref{item:dim weight space lema:g idempotent}. If $\#\cE_g=\dim R[e]$, the statement follows from \eqref{item:dim weight space lema:g idempotent} and Lemma \ref{lema:g idempotent2} \eqref{item:primitive cota igual 2 lema:g idempotent}.

\eqref{item:particion lema:g idempotent} is clear. \eqref{item:otra primitive cota lema:g idempotent} follows from \eqref{item:dim weight space lema:g idempotent} and \eqref{item:particion lema:g idempotent} since
$$R[e]\delta_g=\oplus_{i\in I} R[e]\be_i=\oplus_{L\in\Irr A}\oplus_{\be_i\in\supp_{\cE_g}L}R[e]\be_i.$$
\end{proof}

In some cases, the simple $A$-modules can be distinguished by their weight spaces.
\begin{lema}\label{item:alcanza ver supp para iso lema:g idempotent}
Let $g\in G$ and $\cE_g=\{\be_i\}_{i \in I}$ be a complete set of orthogonal primitive $g$-idempotents. Assume that $top(A\be_i)$ and $top(A\be_j)$ are not isomorphic as $\ku^G$-modules for all $i\neq j$. Let $L$ be a simple $A$-module.
Then $L\simeq top(A\be_i)$ as $A$-modules if and only if $L\simeq top(A\be_i)$ as $\ku^G$-modules.
\end{lema}

\begin{proof}
If $L\simeq top(A\be_i)$ as $\ku^G$-modules, then $g\in\supp L$. Hence $L\simeq top(A\be_j)$ for some $j$. Then $i=j$ because $top(A\be_i)$ and $top(A\be_j)$ are not isomorphic as $\ku^G$-modules for $i\neq j$. The other implication is obvious.
\end{proof}

For each $g\in G$, let $\cE_g$ be a complete set of orthogonal primitive $g$-idempotents. If $\be,\tilde{\be}\in\cE_g$ and $\be A\tilde{\be}\neq0$, it is said that $\be$ and $\tilde{\be}$ are {\it linked}. This is an equivalence relation \cite[Definition 55.1]{CR}. Let $\cE_g=\bigcup_{i\in I_g}B_i$ be the corresponding partition. The subalgebra $A[e]=R[e]\ku^G$ can be used to compute the simple $A$-modules, see for instance \cite[Theorem 2.7.2]{navo}.

\begin{lema}
Let $g\in G$ and $\cE_g=\bigcup_{i\in I_g}B_i$ be as above. Then $\bigoplus_{\be\in B_i}A[e]\be$ is a subalgebra and a set of representatives of its simple modules is
$$\Irr\bigg(\bigoplus_{\be\in B_i}A[e]\be\bigg)=\bigl\{L[g]:L\in\Irr A\,\mbox{ and }\,B_i\cap\supp_{\cE_g}L\neq\emptyset\bigr\}$$
Moreover as algebras
$$
A[e]=\prod_{g\in G,\,i\in I_g}\,\bigoplus_{\be\in B_i}A[e]\be\biggr..
$$
\end{lema}

\begin{proof}
By \eqref{eq:regla de commutacion}, $\be\tilde{\be}=0=\tilde{\be}\be$ if either $\be\in\cE_g$ and $\tilde{\be}\in\cE_{h}$ with $g\neq h$ or $\be,\tilde{\be}\in\cE_g$ but are not linked. Clearly, $B_i$ is a complete set of orthogonal primitive idempotents of $\bigoplus_{\be\in B_i}A[e]\be$. Also $top(A[e]\be)=L[g]$ since $L[g]=top(A\be)[g]=\overline{A[e]\be}$ for all $\be\in\cE_g$.
\end{proof}

For $g\in G$, we define the linear map $\chi_g:A\longrightarrow \ku$ by
\begin{align}\label{eq:chi g}
\chi_g(r f)=\e(r)f(g)\quad\forall\,r f\in A=R\,\ku^G.
\end{align}
If $\chi_g$ is an algebra map, then $\ku_g$ is also an $A$-module.

\begin{lema}
Let $G$ be a finite group, $A$ a finite-dimensional copointed Hopf algebra over $\ku^G$ with diagram $R\in\ydgdual$ and $\chi\in G(A^*)$. If $R$ is generated by $R^\times$ as an algebra, then $\chi=\chi_g$ for some $g\in G$. Moreover, the map
$$G(A^*)\longrightarrow G,\quad\chi_g\longmapsto g\quad$$
is an injective group homomorphism.

In particular, if $R$ is a Nichols algebra, then $R$ is generated by $R^\times$.
\end{lema}

\begin{proof}
Let $g\in G$ such that $\chi(f)=f(g)$ for all $f\in\ku^G$. By \eqref{eq:regla de commutacion}, $\chi(R^\times)=0$. Hence $\chi=\chi_g$.

Let $\chi_g,\chi_h\in G((A^*)$ for some $g,h\in G$. Then $\chi_g*\chi_h$ is an algebra map and $\chi_g*\chi_h(f)=f(gh)$ for all $f\in\ku^G$. Hence $\chi_g*\chi_h=\chi_{gh}$ and $G(A^*)\longrightarrow G$, $\chi_g\longmapsto g$ is an injective group homomorphism.

Finally, if $R$ is a Nichols algebra, then $R$ is generated by $R_{[1]}$. Moreover, $R_{[1]}\subset R^\times$ by \cite[Lemma 3.1 (f)]{AV}. In particular, $R$ is generated by $R^\times$.
\end{proof}

\begin{example}\label{ex:rep of bozonization copointed}
Let $V\in\ydgdual$ with finite-dimensional Nichols algebra $\BV(V)$. Then $\{\delta_g:g\in G\}$ is a complete set of orthogonal primitive idempotents of $\BV(V)\#\ku^G$ and therefore $\{\ku_g:g\in G\}$ are its simple modules.
\end{example}

Let $\int_A^r$ (resp. $\int_A^l$) denote the space of right (resp. left) integrals, see for example \cite{mongomeri}. Since $A$ is finite-dimensional, the space of right (left) integrals is one-dimensional. Let $t\in\int_{A}^r$. Then there exists a unique $\alpha\in G(A^*)$, called {\it the distinguished group-like element}, such that $at=\alpha(a)t$ for all $a\in A$.

\begin{lema}\label{le:supp top and supp soc}
Let $G$ be a finite group, $A$ a finite-dimensional copointed Hopf algebra over $\ku^G$ and $\alpha\in G(A^*)$ the distinguished group-like element. Assume that there is $g\in G$ such that $\alpha(f)=f(g)$ for all $f\in\ku^G$. Hence
$$
\supp(soc(A\be))=g\supp(top(A\be))
$$
for any primitive idempotent $\be\in A$.

In particular, $\int^l_A=soc(A\be_{g^{-1}})\subset R[g]\be_{g^{-1}}$ where $\be_{g^{-1}}$ is the primitive $g^{-1}$-idempotent such that $top(A\be_{g^{-1}})\simeq\ku_{g^{-1}}$ as $\ku^G$-modules.
\end{lema}

\begin{proof}
Let $\eta:A\rightarrow A$ be the Nakayama automorphism. If $M$ is an $A$-module, then $\overline{M}$ denotes the vector space $M$ with action $a\cdot m=\eta^{-1}(a)m$  for all $a\in A$, $m\in M$. By \cite[Lemma 1.5]{FMS},
$$
\eta^{-1}(\delta_t)=\langle\alpha^{-1}, \mathcal{S}^{2}(\delta_t)_1\rangle \mathcal{S}^{2}(\delta_t)_2=\delta_{gt}
$$
for all $t\in G$. Therefore $\overline{M}[h]=M[gh]$ for all $h\in G$. By \cite[Lemma 2]{NeSc}, $top(A\be)=\overline{soc(A\be)}$ and hence $\supp(soc(A\be))=g\supp(top(A\be))$.

In particular, we obtain that $\int^l_A=soc(A\be_{g^{-1}})\subset R[g]\be_{g^{-1}}$, the inclusion follows from \eqref{eq:pesos del verma}.
\end{proof}

We include the next lemma for completeness.
\begin{lema}\label{le:to build idempotents}
Let $A$ be an algebra and $a_1, ...,a_n$ be idempotents of $A$ such that $a_ia_j=a_ja_i$ for all $i,j=1, ..., n$. Set
$$
\be_i=a_i+a_i\sum_{\ell=1}^{i-1}(-1)^{\ell}\sum_{1\leq j_1<\cdots<j_{\ell}\leq i-1} a_{j_1}\cdots a_{j_{\ell}}.
$$
Then $\be_i\be_j=\delta_{i,j}\be_i$ for all $i,j=1, ..., n$.
\end{lema}

\begin{proof}
For $j<i$, we write
\begin{align*}
\be_i=a_i+a_i\sum_{\ell=1}^{i-1}&(-1)^{\ell}\sum_{\substack{1\leq j_1<\cdots<j_{\ell}\leq i-1 \\ j_s\neq j}} a_{j_1}\cdots a_{j_{\ell}}\\
&+a_i\sum_{\ell=1}^{i-1}(-1)^{\ell}\sum_{\substack{1\leq j_1<\cdots<j_{\ell}\leq i-1 \\ j_s=j \,\mbox{{\tiny for some $s$}}}} a_{j_1}\cdots a_{j_{\ell}}.
\end{align*}
Then $a_j\be_i=0$ and hence $\be_j\be_i=\delta_{i,j}\be_i$ for all $i,j=1, ..., n$.
\end{proof}
The order of the set $\{a_i\}$ alters the result of the above lemma. Moreover, it can produce $\be_i=0$ for some $i$. For example: $\{1,a\}$ and $\{a,1\}$ with $a$ an idempotent.

\subsection{Quasitriangular copointed Hopf algebras}\label{subsec:quasitriangular}

\quad Let $G$ be a non-abelian group and $A$ be a quasitriangular finite-dimensional copointed Hopf algebra over $\ku^G$ with $R$-matrix $Q\in A\ot A$, that is, $Q$ is an invertible element which satisfies \cite[(QT.1)--(QT.4)]{radford} and
\begin{equation}\label{eq:def quasi}
Q\Delta(x)=\Delta^{cop}(x)Q,\quad\mbox{for all $x\in A$.}
\end{equation}

Let $(A_Q, Q)$ be its unique minimal subquasitriangular Hopf algebra \cite[p. 292]{radford}. Then $A_Q=HB$ with Hopf subalgebras $H,B\subseteq A$ such that $B\simeq H^{*\cop}$ by \cite[Proposition 2 and Theorem 1]{radford}.

\begin{lema}\label{le:quasi copointed}
$H$, $B$ and $A_Q$ are pointed Hopf algebras over abelian groups. Moreover, $A_Q$ is neither a group algebra nor the bosonization of its diagram with $G(A_Q)$.
\end{lema}

\begin{proof}
Since $H_{[0]}=H\cap A_{[0]}$ and $B_{[0]}=B\cap A_{[0]}$, there are group epimorphisms $G\rightarrow G_H$ and $G\rightarrow G_B$ such that $H_{[0]}=\ku^{G_H}$ and $B_{[0]}=\ku^{G_B}$. Then there is an epimorphism of Hopf algebras $B\overset{\simeq}{\longrightarrow} H^{*\cop}\longrightarrow\ku G_H$. By \cite[Corollary 5.3.5]{mongomeri}, the restriction $B_{[0]}=\ku^{G_B}\rightarrow\ku G_H$ is surjective. Thus $G_H$ is an abelian group. {\it Mutatis mutandis}, we see that $G_B$ is also an abelian group. Hence $H$ and $B$ are generated by skew-primitives and group-likes elements by \cite[Theorem 2]{An} and therefore is also $A_Q=HB$. Then $A_Q=HB$, $H$ and $B$ are pointed Hopf algebras over abelian groups. Set $\Gamma=G(A_Q)$.

Now we assume $A_Q=\ku \Gamma$ and let $\delta_g\in\ku^G\setminus\ku \Gamma$. It must hold $Q\Delta(\delta_g)=\Delta^{cop}(\delta_g)Q$ by \eqref{eq:def quasi}. However, this is not possible since $Q$ is invertible and $\ku^G$ is commutative but not cocommutative. Then $A_Q\neq\ku \Gamma$.

Finally, we assume that $A_Q=\BV(V)\#\ku\Gamma$ where $\BV(V)$ is the diagram of $A_Q$ which is a Nichols algebra by \cite[Theorem 2]{An}. Let $Q_0\in\ku\Gamma\ot\ku\Gamma$ and $Q^+\in \BV(V)^+\#\ku\Gamma\ot\ku\Gamma+\ku\Gamma\ot\BV(V)^+\#\ku\Gamma$ such that $Q=Q_0+Q^+$. Then $Q_0$ is invertible since $Q$ is so and $\BV(V)^+$ is nilpotent. If $\delta_g\in \ku^G\setminus\ku\Gamma$, then it must hold $Q_0\Delta(\delta_g)=\Delta^{cop}(\delta_g)Q_0$ by \eqref{eq:def quasi}. As above, this is not possible. Therefore $A_Q\not\neq\BV(V)\#\ku\Gamma$.
\end{proof}

\section{The tetrahedron rack and their associated algebras}\label{sec:pre}

\quad Let $\F_4$ be the finite field of four elements and $\omega\in\F_4$ such that $\omega^2+\omega+1=0$. The tetrahedron rack is the affine rack $\Aff(\F_4,\omega)$. That is, the set $\F_4$ with operation $a\rhd b=\omega b+\omega^2a$.

\smallbreak

Let $(\cdot, g, \chi_G)$ be a \emph{faithful principal YD-realization} of $(\Aff(\F_4,\omega),-1)$ over a finite group $G$ \cite[Definition 3.2]{AG3}. Recall that
\begin{itemize}
\item $\cdot$ is an action of $G$ over $\F_4$,
\smallbreak
\item $g:\F_4\to G$ is an injective function such that $g_{h\cdot i} = hg_ {i}h^{-1}$
and $g_{i}\cdot j=i\rhd j$ for all $i,j\in\F_4$, $h\in G$
\smallbreak
\item\label{item:1-cocycle} $\chi_G:G\rightarrow\ku^*$ is a multiplicative character such that $\chi_G(g_i)=-1$ for all $i\in\F_4$; we can consider such a $\chi_G$ by \cite[Lemma 3.3(d)]{AG3}.
\end{itemize}
These data define a structure on $V=\ku\{x_i\}_{i\in\F_4}$ of Yetter-Drinfeld module over $\ku^G$ via
\begin{align}\label{eqn:yetter-drinfeld-dual}
\delta_t\cdot x_i=\delta_{t,g_i^{-1}}x_i\quad\mbox{ and }\quad
\lambda(x_i)=\sum_{t\in
G}\chi_G(t^{-1})\delta_t\ot x_{t^{-1}\cdot  i}\,\mbox{ $\forall t\in G$, $i\in X$}.
\end{align}
We obtain \eqref{eqn:yetter-drinfeld-dual} using the fact that the categories $\ydgdual$ and $\ydg$ are braided equivalent \cite[Proposition 2.2.1]{AG1}, see \cite[Subsection 3.2]{agustinvay} for details.

We denote by $G'$ the subgroup of $G$ generated by $\{g_i\}_{i\in\F_4}$. Then $G'$ is a quotient of the {\it enveloping group of} $\Aff(\F_4,\omega)$ \cite{EG,J}:
\begin{align*}%\label{eq:enveloping group}
G_{\Aff(\F_4,\omega)}=\langle g_0,g_1,g_\omega,g_{\omega^2}\,|\, g_ig_j=g_{i\rhd j}g_i,\, i,j\in \F_4\rangle.
\end{align*}

Let $m\in\N$. We denote by $C_m$ the cyclic group of order $m$ generated by $t$. The semidirect product group $\F_4\rtimes_\omega C_{6m}$ is given by $t\cdot i=\omega i$ for all $i\in\F_4$.

\begin{example}\label{ex:mk afin realization}
\cite[Proposition 4.1]{agustinvay} Let $k,m\in\N$, $0\leq k< m$. The {\it $(m,k)$-affine realization} of $(\Aff(\F_4,\omega),-1)$ over $\F_4\rtimes_\omega C_{6m}$ is defined by
\begin{itemize}
\item $g:\F_4\to \F_4\rtimes_\omega C_{6m}$, $i\mapsto g_i=(i, t^{6k+1})$;
\smallbreak
\item $\cdot:\F_4\rtimes_\omega C_{6m}\to \F_4$ is $h\cdot i=j, \text{ if }hg_ih^{-1}=g_j$;
\smallbreak
\item $\chi_{\F_4\rtimes_\omega C_{6m}}:\F_4\rtimes_\omega C_{6m}\longmapsto\ku^*$, $(j, t^s)\mapsto(-1)^s$, $\forall i,j\in A$, $s\in\N$.
\end{itemize}
\end{example}

\subsection{A Nichols algebra over $\Aff(\F_4,\omega)$}\label{subsec:nichols f4 omega}

\quad From now on, {\it we fix a faithful principal YD-realization $(\cdot, g, \chi_G)$ over a finite group $G$ of $(\Aff(\F_4,\omega),-1)$. Let $V\in\ydgdual$ be as in \eqref{eqn:yetter-drinfeld-dual}.}

\smallbreak

In \cite[Subsection 3.1]{agustinvay} it was discussed how braided functors modify the Nichols algebras. As a consequence the defining relations of the Nichols algebra $\BV(V)$ were calculated in \cite{agustinvay} using previous results of \cite{grania1} for the pointed case. Namely, $\BV(V)$ is the quotient of $T(V)$ by the ideal $\cJ(V)$ generated by
\begin{align}\label{eq:F4 omega}
&x_i^2, \quad x_j\,x_i+x_i\,x_{(\omega+1)i+\omega j}+x_{(\omega+1)i+\omega j}\,x_j\quad\mbox{$\forall i,j\in\F_4$ and}\\
\noalign{\smallskip}
\label{eq:z def f4 omega}
&z:=(x_\omega x_0x_1)^2+(x_1x_\omega x_0)^2+(x_0x_1x_\omega)^2.
\end{align}
In fact, \cite[Proposition 4.4 (b)]{agustinvay} states that $\cJ(V)$ is generated by the elements in \eqref{eq:F4 omega} and $z'_{(-1,4, \omega)}=(x_\omega x_{\omega^2}x_0)^2+(x_1 x_{\omega^2}x_\omega)^2+(x_0 x_{\omega^2}x_1)^2$. An straightforward computation shows that $z-z'_{(-1,4, \omega)}$ belongs to the ideal generated by the elements in \eqref{eq:F4 omega}.  Hence, we can take $z$ as a generator of $\cJ(V)$ instead of $z'_{(-1,4, \omega)}$.

Let $\B$ be the subset of $\BV(V)$ consisting of all possible words $m_1m_2m_3m_4m_5$ such that $m_i$ is an element in the $i$th row of the next list
\begin{align*}
1, &\, x_0, \\
1, &\, x_1,\,  x_1x_0,\\
1, &\, x_\omega x_0x_1,\\
1, &\, x_\omega,\, x_\omega x_0,\\
1, &\, x_{\omega^2}.
\end{align*}

By \eqref{eqn:yetter-drinfeld-dual} the weight of a monomial $x_{i_1}\cdots x_{i_\ell}\in T(V)$ is $g^{-1}_{i_1}\cdots g^{-1}_{i_\ell}$. Set $g_{top}=g_0^{-1}g_1^{-1}g_0^{-1}g_{\omega}^{-1}g_0^{-1}g_1^{-1}g_{\omega}^{-1}g_0^{-1}g_{\omega^2}^{-1}$ and
$$
m_{top}=x_0x_1x_0x_{\omega}x_0x_1x_{\omega}x_0x_{\omega^2}\in\B[g_{top}].
$$

\begin{lema}
The set $\B$ is a basis of $\BV(V)$ and $m_{top}$ is an integral.
\end{lema}

\begin{proof}
The faithful principal YD-realization $(\cdot, g, \chi_G)$ over $G$ of $(\Aff(\F_4,\omega),-1)$ also defines a Yetter-Drinfeld module $W\in\ydg$ with basis $\{y_i\}_{i\in\F_4}$, see for instance \cite[(7)]{agustinvay}. By \cite[Theorem 6.15]{AG-adv}, the ideal defining the Nichols algebra $\BV(W)$ is generated by
\begin{align*}
&y_i^2, \quad y_i\,y_j+y_{(\omega+1)i+\omega j}\,y_i+y_j\,y_{(\omega+1)i+\omega j}\quad\mbox{$\forall i,j\in\F_4$ and}\\
\noalign{\smallskip}
&(y_\omega y_1y_0)^2+(y_0y_\omega y_1)^2+(y_1y_0y_\omega)^2.
\end{align*}

Let $\phi:W\rightarrow V$ be the linear map defined by $\phi(y_0)=x_1$, $\phi(y_1)=x_0$, $\phi(y_\omega)=x_\omega$ and $\phi(y_{\omega^2})=x_{\omega^2}$. By \eqref{eq:F4 omega} and \eqref{eq:z def f4 omega}, $\phi$ induces an algebra isomorphism $\phi':\BV(W)\rightarrow\BV(V)$. Also, \cite[Theorem 6.15]{AG-adv} gives a basis $B$ of $\BV(W)$ which consists of all possible words $m_1m_2m_3m_4m_5$ such that $m_i$ is an element in the $i$th row of the next list
\begin{align*}
1, &\, x_1, \\
1, &\, x_0,\,  x_0x_1,\\
1, &\, x_\omega x_0x_1,\\
1, &\, x_\omega,\, x_\omega x_0,\\
1, &\, x_{\omega^2}.
\end{align*}
Then $\phi'(B)$ is a basis of $\BV(V)$. Since $x_1x_0x_1=x_0x_1x_0$ in $\BV(V)$ by \eqref{eq:F4 omega}, $\B$ also is a basis of $\BV(V)$.

Finally, the space of integrals of a finite-dimensional Nichols algebra is the homogeneous component of bigger degree, see for instance \cite[p. 227]{AG-adv}. Therefore $m_{top}$ is an integral.
\end{proof}

\smallbreak

\begin{lema}\label{le:Be contained in}
Let $G$ be a finite group with a faithful principal YD-realization $(\cdot, g, \chi_G)$ of $(\Aff(\F_4,\omega),-1)$. Then
\begin{enumerate}\renewcommand{\theenumi}{\alph{enumi}}\renewcommand{\labelenumi}{
(\theenumi)}
\item\label{item: about supp le:Be contained in} $\supp \BV(V)=\supp\B\subset G'$.
\smallbreak
\item\label{item: epi le:Be contained in} $G'\longmapsto\F_4\rtimes_\omega C_6$, $g_i\mapsto (i, t)$ is an epimorphism of groups.
\smallbreak
\item\label{item: Be contained in le:Be contained in} If $z\in T(V)[e]$, then $\B[e]=\{1,b_1,b_2,b_3,b_4,b_5\}$ where
\begin{align*}
&b_1=x_0x_1x_0x_{\omega}x_0x_{\omega^2},&&b_2=x_0x_{\omega}x_0x_1x_{\omega}x_{\omega^2},& &b_3=x_1x_0x_{\omega}x_0x_1x_{\omega^2}\\
&b_4=x_1x_{\omega}x_0x_1x_{\omega}x_0,&&b_5=x_0x_{1}x_{\omega}x_0x_1x_{\omega}.&&&
\end{align*}
\smallbreak
\item\label{item: decomposition cosimples le:Be contained in} Let $y=\sum_{i\in\F_4}x_i$ and $U=\ku\{x_0-x_1,x_0-x_\omega,x_0-x_{\omega^2}\}$. Then $\ku y$ and $U$ are simple $\ku^G$-comodules such that $V=\ku y\oplus U$.
\end{enumerate}
\end{lema}

\begin{proof}
\eqref{item: about supp le:Be contained in} holds since the elements of $\B$ are $\ku^G$-homogeneous and $\BV(V)$ is a $\ku^G$-module algebra.

\eqref{item: epi le:Be contained in} By \cite[Lemma 1.9 (1)]{AG-adv}, the quotient of $G'$ by its center $\mathcal{Z}(G')$ is isomorphic to $\operatorname{Inn}_{\rhd}\Aff(\F_4,\omega)=\F_4\rtimes_\omega C_3$ via $\overline{g_i}\mapsto (i, t)$, $i\in\F_4$. Then $G'/(\mathcal{Z}(G')\cap\ker\chi_G)\simeq\F_4\rtimes_\omega C_3\times C_2\simeq\F_4\rtimes_\omega C_{6}$.

\eqref{item: Be contained in le:Be contained in} If $z\in\B[e]$, then $\{1,b_1,b_2,b_3,b_4,b_5\}\subseteq\B[e]$ since $g_ig_j=g_{i\rhd j}g_i$. Let $w=x_{i_1}\cdots x_{i_s}\in\B[e]$. Applying the epimorphism of \eqref{item: epi le:Be contained in} to the weight of $w$, we see that $w=1$ or $s=6$. If $w\neq1$, we can check that $w=b_i$ for some $i$.

\eqref{item: decomposition cosimples le:Be contained in} is equivalent to prove that $\ku y$ and $U$ are simple $\ku G$-modules via the action $g\cdot x_i=\chi_G(g)\, x_{g\cdot i}$, $i\in\F_4$. Clearly, $\ku y$ and $U$ are $\ku G$-submodules and $\ku y$ is $\ku G$-simple. Moreover, it is an straightforward computation to show that $U$ is $\ku G'$-simple and therefore $\ku G$-simple.
\end{proof}

\subsection{Copointed Hopf algebras over $\Aff(\F_4,\omega)$}\label{sect:hopf over f4}

\quad The copointed Hopf algebras over $\ku^G$ whose infinitesimal
braiding arises from a principal YD-realization of the affine rack $\Aff(\F_4,\omega)$ with the
constant 2-cocycle $-1$ are classified in \cite{agustinvay} as follows.

By \eqref{eqn:yetter-drinfeld-dual} the smash product Hopf algebra $T(V)\#\ku^G$ is defined by
\begin{align}\label{eq:TV smash ku a la G}
\begin{split}
\delta_tx_i&=x_i\delta_{g_i\,t}\quad\mbox{and}\\
\Delta(x_i)&=x_i\ot 1 +\sum_{t\in G}\chi_G(t)\delta_{t^{-1}}\ot x_{t\cdot
i}\quad\mbox{$\forall t\in G$, $i\in X$}.
\end{split}
\end{align}

\begin{definition}\label{def:Lglambda}
Let $\lambda\in\ku$ and assume $z\in T(V)[e]$. The Hopf algebra $\cA_{G,\lambda}$ is the quotient of $T(V)\#\ku^G$ by the ideal generated by \eqref{eq:F4 omega} and
$z-f$ where
$$
f=\lambda(1-\chi_z^{-1})\quad\mbox{and}\quad \chi_z=\chi_G^6.
$$
Notice that if either $\lambda=0$ or $\chi_z=1$, then $\cA_{G,\lambda}=\BV(V)\#\ku^G$.
\end{definition}

\smallbreak

The next theorem is \cite[Main theorem 2 and Theorem 7.5]{agustinvay}.

\begin{Theorem}\label{thm:lifting over affines grandes}
Let $A$ be a copointed Hopf algebra over $\ku^G$ whose infinitesimal
braiding arises from a principal YD-realization of the affine rack $\Aff(\F_4,\omega)$ with the
constant 2-cocycle $-1$.
\begin{enumerate}\renewcommand{\theenumi}{\alph{enumi}}\renewcommand{\labelenumi}{
(\theenumi)}
\item\label{item:no deforma g gen thm:lifting over affines grandes} If $G=G'$, then $A\simeq\BV(V)\#\ku^G$.
\smallbreak
\item\label{item:no deforma no e thm:lifting over affines grandes} If $z\in
T(V)^\times$, then $A\simeq\BV(V)\#\ku^G$.
\smallbreak
\item\label{item:quadratic thm:lifting over affines grandes} If $z\in T(V)[e]$, then
$A\simeq\cA_{G,\lambda}$ for some $\lambda\in\ku$.
\smallbreak
\item\label{item:Lglambda is cocycle deformation thm:lifting over affines grandes}
$\cA_{G,\lambda}$ is a cocycle deformation of $\cA_{G,\lambda'}$,
for all $\lambda, \lambda'\in\ku$.
\smallbreak
\item\label{item:Lglambda lifting thm:lifting over affines grandes}
$\cA_{G,\lambda}$ is a lifting of $\BV(V)$ over $\ku^G$
for all $\lambda, \lambda'\in\ku$.
\smallbreak
\item\label{item:Lglambda iso class thm:lifting over affines grandes}
$\cA_{G,\lambda}\simeq\cA_{G,1}\not\simeq\cA_{G,0}$ for all $\lambda\in\ku^*$.\qed
\end{enumerate}
\end{Theorem}

We are specially interested in the case that $\cA_{G,\lambda}$ is not isomorphic to $\BV(V)\#\ku^G$. The next faithful principal YD-realization gives such a $\cA_{G,\lambda}$.

\begin{example}\label{ex:para spherical}
Suppose that $m\mid6 k+1$. Let $G_1$ be a finite group with a multiplicative character
$\chi_{G_1}:G_1\rightarrow\ku^*$ such that $\chi_{G_1}^{6}\neq1$. Then the $(m,k)$-affine realization, recall Example \ref{ex:mk afin realization}, is extended to a principal YD-realization over $G=\F_4\rtimes_\omega C_{6m}\times G_1$ setting $G_1\cdot i=i$ and $\chi_{G}=\chi_{\F_4\rtimes_\omega C_{6m}}\times\chi_{G_1}$. Note that $z\in T(V)[e]$ and $\chi_z=\chi_G^{6}\neq1$.
\end{example}

The next example will be necessary in Lemma \ref{le:delta h en ker}.
\begin{example}\label{ex:lema14}
Let $G'\leq G_1\leq G$ be finite groups. If $(\cdot, g, \chi_G)$ is a faithful principal YD-realization of $(\Aff(\F_4,\omega),-1)$ over $G$, then $(\cdot, g, (\chi_G)_{|G_1})$ is a faithful principal YD-realization of $(\Aff(\F_4,\omega),-1)$ over $G_1$. For instance, $G_1=\ker\chi_z$.
\end{example}

We think of $\cA_{G,\lambda}$ as an algebra presented by generators $\{x_i,\delta_g:i\in\F_4,\,g\in G\}$ and relations:
\begin{align}\label{eq:relations of Aglambda}
\notag\delta_gx_i=x_i\delta_{g_ig},\quad x_i^2&=0,\quad\delta_g\delta_h=\delta_{g}(h)\delta_g,\quad1=\sum_{g\in G}\delta_g,\\
x_0x_\omega+x_\omega x_1+x_1x_0&=0=x_0x_{\omega^2}+x_{\omega^2}x_\omega+x_{\omega}x_0,\\
\notag x_1x_{\omega^2}+x_0x_1+x_{\omega^2}x_0&=0=x_{\omega}x_{\omega^2}+x_1x_{\omega}+x_{\omega^2}x_1\quad\mbox{and}\\
\notag x_\omega x_0x_1x_\omega x_0x_1+x_1x_\omega& x_0x_1x_\omega x_0+x_0x_1x_\omega x_0x_1x_\omega=f,
\end{align}
for all $i\in\F_4$ and $g\in G$. Since $\chi_z(g_i)=1$, it holds that
\begin{align}\label{eq:uno menos chi es central}
f\,x_i=x_i\,f\quad\forall i\in\F_4.
\end{align}
A basis for $\cA_{G,\lambda}$ is  $\A=\{x\delta_g | x \in \B, g \in G \}$ and a basis for the Verma module $M_g$ is
$\MM=\{x_{i_1}\cdots x_{i_s}\delta_g\in\B\delta_g\}$.

\begin{prop}\label{prop:AGlambda notquasi}
$\cA_{G,\lambda}$ is not quasitriangular.
\end{prop}

\begin{proof}
Let $A$ be a pointed Hopf subalgebra of $\cA_{G,\lambda}$ with abelian group of group-like elements. By Lemma \ref{le:quasi copointed}, the proposition follows if we show that $A$ is either a group algebra or the bosonization of its diagram with $G(A)$.

Note that $A$ is generated by skew-primitives and group-likes elements by \cite[Theorem 2]{An}.

Let $y=\sum_{i\in\F_4}x_i$. The space of skew-primitives of $\cA_{G,\lambda}$ is $\ku G(\cA_{G,\lambda})\oplus\ku\{yg\,|\,g\in G(\cA_{G,\lambda})\}$ by Lemma \ref{le:Be contained in} \eqref{item: decomposition cosimples le:Be contained in}. Also, $y^2=0$ by \eqref{eq:relations of Aglambda}. Hence $A=G(A)$ or $A=(\ku[y]/\langle y^2\rangle)\#\ku G(A)$.
\end{proof}

\section{Representation theory of $\cA_{G,\lambda}$}\label{sec:rep of Aglambda}
\

Let $(\cdot, g, \chi_G)$ be a faithful principal YD-realization  of $(\Aff(\F_4,\omega),-1)$ over a fixed finite group $G$. Let $V\in\ydgdual$ be as in \eqref{eqn:yetter-drinfeld-dual}.

We are interested in the representation theory of the liftings of the Nichols algebra $\BV(V)$ over $\ku^G$. By Theorem \ref{thm:lifting over affines grandes}, these liftings are the Hopf algebras $\cA_{G,\lambda}$, $\lambda\in\ku$, recall Definition \ref{def:Lglambda}. We begin by classifying the simple modules.

\smallbreak

If $\cA_{G,\lambda}$ is isomorphic to the bosonization $\BV(V)\#\ku^G$, then the simple modules are the one-dimensional modules $\ku_g$, $g\in G$, where the Nichols algebra acts by zero, see Example \ref{ex:rep of bozonization copointed}.

\smallbreak

From now on, {\it we fix $\lambda\in\ku^*$ and assume that $\cA_{G,\lambda}$ is not isomorphic to the bosonization $\BV(V)\#\ku^G$.}

In this case, $z\in T(V)[e]$ and $\chi_z\neq1$ by Theorem \ref{thm:lifting over affines grandes} and Definition \ref{def:Lglambda}. Let $f=\lambda(1-\chi_z^{-1})$ be as in Definition \ref{def:Lglambda}. For $g\in G\setminus\ker\chi_z$, we define
\begin{align*}
\be_1^g&=-\frac{1}{f(g)}b_1\delta_g,&\be_2^g&=-\frac{1}{f(g)}b_2\delta_g,&\be_3^g&=\frac{1}{f(g)}b_3\delta_g,\\
\be_4^g&=\frac{1}{f(g)}(b_4-b_3)\delta_g,&\be_5^g&=\frac{1}{f(g)}(b_5+b_1)\delta_g&&\mbox{and}\\
\be_6^g&=\delta_g+\frac{1}{f(g)}(b_2-b_4-b_5)\delta_g,&
\end{align*}
where $b_1,b_2,b_3,b_4,b_5\in\cA_{G,\lambda}$ are as in Lemma \ref{le:Be contained in} \eqref{item: Be contained in le:Be contained in}.

\begin{lema}\label{le:cE}
A complete set of orthogonal primitive idempotents of $\cA_{G,\lambda}$ is
$$
\cE:=\bigl\{\delta_h,\,\be_1^g,\,\be_2^g,\,\be_3^g,\,\be_4^g,\,\be_5^g,\,\be_6^g\,|\,h\in\ker\chi_z,\,g\in G\setminus\ker\chi_z\bigr\}.
$$
\end{lema}

\begin{proof}
By Lemma \ref{le:Be contained in} \eqref{item: Be contained in le:Be contained in}, $\{b_i\delta_g|1\leq i\leq 6\}$ is a basis of $\BV(V)[e]\delta_g$ for all $g\in G$. By \eqref{eq:relations of Aglambda} and \eqref{eq:uno menos chi es central}, it holds that:
\begin{align}
\notag b_1^2&=- b_1f,&b_1b_2&=0,&b_1b_3&=0,&b_1b_4&=0,&b_1b_5&=b_1f,&\\
\notag b_2b_1&=0,&b_2^2&=-b_2f,&b_2b_3&=0,&b_2b_4&=0,&b_2b_5&=0,&\\
\label{eq:bibj} b_3b_1&=0,&b_3b_2&=0,&b_3^2&=b_3f,&b_3b_4&=b_3f,&b_3b_5&=0,&\\
\notag b_4b_1&=0,&b_4b_2&=0,&b_4b_3&=b_3f,&b_4^2&=b_4f,&b_4b_5&=0,&\\
\notag b_5b_1&=b_1f,&b_5b_2&=0,&b_5b_3&=0,&b_5b_4&=0,&b_5^2&=b_5f.&
\end{align}

Therefore $\cE_h=\{\delta_h\}$ is a complete set of orthogonal primitive $h$-idem\-po\-tents for all $h\in\ker\chi_z$. If $g\in G\setminus\ker\chi_z$, we apply Lemma \ref{le:to build idempotents} to the ordered set
$$
\biggl\{-\frac{1}{f(g)}b_1\delta_g,\,-\frac{1}{f(g)}b_2\delta_g,\,\frac{1}{f(g)}b_3\delta_g,\frac{1}{f(g)}b_4\delta_g,\,\frac{1}{f(g)}b_5\delta_g,\,\delta_g\biggr\}
$$
and hence $\cE_g=\{\be_i^g|1\leq i\leq6\}$ is a complete set of orthogonal primitive $g$-idempotents. Then $\cE=\cup_{g\in G}\cE_g$ is a complete set of orthogonal primitive idempotents.
\end{proof}

Let $M$ be an $\cA_{G,\lambda}$-module. Since $\cA_{G,\lambda}$ is a quotient of $T(V)\#\ku^{G}$, $M$ is also a $T(V)\#\ku^{G}$-module. Moreover, $M$ is a $T(V)\#\ku^{\ker\chi_z}$-module  if $\supp M\subseteq\ker\chi_z$ since $T(V)\#\ku^{\ker\chi_z}$ is a subalgebra of $T(V)\#\ku^{G}$, cf. Example \ref{ex:lema14}.

\begin{lema}\label{le:delta h en ker}
Let $h\in\ker\chi_z$.
\begin{enumerate}\renewcommand{\theenumi}{\alph{enumi}}\renewcommand{\labelenumi}{
(\theenumi)}
\item\label{item:equiv of cat le:delta h en ker} If $M$ is an $\cA_{G,\lambda}$-module with $\supp M\subseteq\ker\chi_z$, then $M$ is a module over $\BV(V)\#\ku^{\ker\chi_z}$.
\smallbreak
\item\label{item:Mh is free le:delta h en ker} $M_h$ is a free $\BV(V)$-module of rank $1$ generated by $\delta_h$.
\smallbreak
\item\label{item:chi de h es alg le:delta h en ker} $\chi_h:\cA_{G,\lambda}\rightarrow\ku$ is an algebra map.
\smallbreak
\item\label{item:top y soc le:delta h en ker} $top(M_h)\simeq\ku_h$ and $soc(M_h)\simeq\ku_{g_{top}h}$.
\smallbreak
\item\label{item:integral le:delta h en ker} $\int_{\cA_{G,\lambda}}^l=soc(M_{g_{top}^{-1}})$ and $\chi_{g_{top}}$ is the distinguished group-like element.
\end{enumerate}
\end{lema}

\begin{proof}
\eqref{item:equiv of cat le:delta h en ker} Since $M$ is a $T(V)\#\ku^{\ker\chi_z}$-module, we have to see that the elements in \eqref{eq:F4 omega} and $z$ act by zero over $M$. This is true for the first elements because they are zero in $\cA_{G,\lambda}$. If $h\in\ker\chi_z$, then $f\delta_h=0$ and hence $z\cdot M[h]=f\cdot(\delta_h\cdot M)=0$. \eqref{item:Mh is free le:delta h en ker} follows from \eqref{item:equiv of cat le:delta h en ker}. \eqref{item:chi de h es alg le:delta h en ker} is clear. \eqref{item:top y soc le:delta h en ker} and \eqref{item:integral le:delta h en ker} follows from \eqref{item:Mh is free le:delta h en ker} and Lemma \ref{le:supp top and supp soc}.
\end{proof}

For each $\be_i^g\in\cE$, we set $L_i^g=\cA_{G,\lambda}\be_i^g$.

\begin{lema}\label{le:dim12}
\begin{enumerate}\renewcommand{\theenumi}{\alph{enumi}}\renewcommand{\labelenumi}{
(\theenumi)}
\item\label{item:son simmples le:dim12} $L_i^g$ is an injective and projective simple module of dimension $12$ for all $\be_i^g\in\cE$.
\smallbreak
\item\label{item:L le:dim12} There exist $\ku^G$-submodules $L_1, \dots, L_6\subset\BV(V)$ such that $\BV(V)=L_1\oplus \cdots\oplus L_6$ and  $L_i^g=L_i\delta_g$ for all $i=1, \dots, 6$ and $g\in G$.
\smallbreak
\item\label{item:acerca del supp le:dim12} $\supp L_i\neq \supp L_j$ and $\supp L_i^g=(\supp L_i)g$ for all $1\leq i,j\leq 6$ and $g\in G$.
\smallbreak
\item\label{item:L isos le:dim12} $L_i^g\simeq L_j^h$ if and only if $(\supp L_i)g=(\supp L_j)h$.
\end{enumerate}
\end{lema}

\begin{proof}
\eqref{item:son simmples le:dim12} Let $v=\overline{\be_i^g}\in top(L_i^g)$. Since $f(g)v=z\cdot v=(x_\omega x_0x_1)^2\cdot v+b_4\cdot v+ b_5\cdot v\neq0$, there are $x_{i_6}$, $\dots, x_{i_1}\in\cA_{G,\lambda}$ such that
$x_{i_\ell}\cdots x_{i_1}\cdot v\neq0$ for all $\ell=1, ..., 6$.

We claim that $\dim top(L_i^g)\geq11$. In fact, if $1\leq \ell<6$, then by \eqref{eq:F4 omega}
\begin{align*}
x_{i_{\ell+1}}x_{i_\ell}\cdots x_{i_1}\cdot v&=\\
 -x_{i_\ell}&\,x_{(\omega+1)i_\ell+\omega i_{\ell+1}}\cdots x_{i_1}\cdot v-x_{(\omega+1)i_\ell+\omega i_{\ell+1}}\,x_{i_{\ell+1}}\cdots x_{i_1}\cdot v\neq0
\end{align*}
and hence $x_{(\omega+1)i_\ell+\omega i_{\ell+1}}\cdots x_{i_1}\cdot v\neq0$ or $x_{i_{\ell+1}}\cdots x_{i_1}\cdot v\neq0$. Applying the epimorphism given by Lemma \ref{le:Be contained in} \eqref{item: epi le:Be contained in}, we find $11$ elements with different weights belong to $top(L_i^g)$. Then $\#\supp top(L_i^g)\geq11$.

\smallbreak

Now, we show that $L_i^g=soc(L_i^g)=top(L_i^g)$ and \eqref{item:son simmples le:dim12} follows. Otherwise, $\dim L_i^g\geq22$ since $\dim top(L_i^g)=\dim soc(L_i^g)$ by \cite[Lemma 58.4]{CR}. But the above claim holds for all $i$ and hence $72=\dim M_g\geq22 + 5\cdot11$, a contradiction.

\eqref{item:L le:dim12} follows from Tables \ref{table:module L1g}--\ref{table:module L6g} in Appendix. \eqref{item:acerca del supp le:dim12} If $G'=\F_4\rtimes C_6$, then $\supp L_i\neq \supp L_j$ by Table \ref{table:pesos} in Appendix and therefore for any $G'$ by Lemma \ref{le:Be contained in} \eqref{item: epi le:Be contained in}. By \eqref{item:L le:dim12}, $\supp L_i^g=(\supp L_i)g$. \eqref{item:L isos le:dim12} follows from \eqref{item:acerca del supp le:dim12} and Lemma \ref{item:alcanza ver supp para iso lema:g idempotent}.
\end{proof}

We consider the product set $\{1,2,3,4,5,6\}\times G$ with the equivalence relation $i\times g\sim j\times h$ if and only if $(\supp L_i)g=(\supp L_j)h$. Let $\mathfrak{X}$ be the set of equivalence classes of $\sim$. We denote by $[i,g]$ the equivalence class of $i\times g$. By Lemma \ref{le:dim12} \eqref{item:L isos le:dim12}, we can define $L_{[i,g]}=L_i^g$.
\begin{Theorem}
Every simple $\cA_{G,\lambda}$-module is isomorphic to either
\begin{align*}
\ku_g&\quad\mbox{for a unique}\quad g\in\ker\chi_z\quad\mbox{or}\\
L_{[i,g]}&\quad\mbox{for a unique}\quad[i,g]\in\mathfrak{X}.
\end{align*}
In particular, there are (up to isomorphism) $|\ker\chi_z|$ one-dimensional simple $\cA_{G,\lambda}$-modules and $\frac{(|G|-|\ker\chi_z|)}{2}$ 12-dimensional simple $\cA_{G,\lambda}$-modules.
\end{Theorem}

\begin{proof}
It follows from Lemmata \ref{le:cE}, \ref{le:delta h en ker} and \ref{le:dim12}.
\end{proof}

\begin{example}
Assume $G'=\F_4\rtimes C_6$ and let $g\in G\setminus\ker\chi_z$. The set $\mathfrak{X}$ is completely defined by the equivalence class $[1, g]$ which is
\begin{align*}
\bigg\{1\times g,\,&2\times(1,t^2)g,3\times(0,t)g,4\times(\omega, t^2)g,\,5\times(1,t)g,\,6\times(\omega,1)g,\,1\times(0,t^3)g\\
&2\times(1,t^5)g,\,3\times(0,t^4)g,\,4\times(\omega, t^5)g,\,5\times(1, t^4)g,\,6\times(\omega, t^3)g\bigg\}.
\end{align*}
Therefore
\begin{align*}
L_{[1,g]}=L_1^g\simeq L_2^{(1,t^2)g}\simeq L_3^{(0,t)g}\simeq L_4^{(\omega,t^2)g}\simeq L_5^{(1,t)g} \simeq L_6^{(\omega,1)g}\simeq\\
\noalign{\smallskip}
L_1^{(0,t^3)g}\simeq L_2^{(1,t^5)g}\simeq L_3^{(0,t^4)g}\simeq L_4^{(\omega,t^5)g}\simeq L_5^{(1,t^4)g} \simeq L_6^{(\omega,t^3)g}.
\end{align*}
Note that $i\times g\sim i\times(0,t^3)g$ for all $i$, hence $L_i^g\simeq L_i^{(0,t^3)g}$.

In fact, $(\supp L_2)(1,t^2)=\supp L_1$, see Tables \ref{table:module L1g} and \ref{table:module L2g}. Then $L_1^g\simeq L_2^{(1,t^2)g}$ by Lemma \ref{le:dim12} \eqref{item:L isos le:dim12}. The other isomorphisms are obtained in the same way.
\end{example}

\subsection{Decomposition of the category of $\cA_{G,\lambda}$-modules}\label{subsec:decomposition}
\

We fix $\lambda\in\ku^*$ and assume that $\cA_{G,\lambda}$ is not isomorphic to the bosonization $\BV(V)\#\ku^G$. Let $I\subset \{1,2,3,4,5,6\}\times G$ be a set of representatives of the equivalence classes of $\sim$. Let $M$ be an $\cA_{G,\lambda}$-module.

If $i\times g\in I$, then $d_{[i,g]}^M=\dim(\be_{i}^{g}\cdot M)$ is the number of composition factors of $M$ which are isomorphic to $L_{[i,g]}$ \cite[Theorem 54.16]{CR}. The number $d_{[i,g]}^M$ can be calculated by Lemma \ref{lema:g idempotent2} \eqref{item:ecdotM lema:g idempotent}. Since $L_{[i,g]}$ is projective and injective by Lemma \ref{le:dim12}, there is a submodule $N\subseteq M$ such that $\supp N\subseteq\ker\chi_z$ and
$$
M=N\oplus\bigoplus_{j\in I}(L_j)^{d_{[i,g]}^M}.
$$
Moreover, $N$ is a $\BV(V)\#\ku^{\ker\chi_z}$-module by Lemma \ref{le:delta h en ker} \eqref{item:equiv of cat le:delta h en ker}.

\subsection{Representation type of $\cA_{G,\lambda}$}
\quad From now on, $\cA_{G,\lambda}$ is any lifting of $\BV(V)$ over $\ku^G$. It can be isomorphic to $\BV(V)\#\ku^G$ or not. Let $\ku_g$ and $\ku_h$ be one-dimensional $\cA_{G,\lambda}$-modules such that $g=g_i^{-1}h\in\ker\chi_z$ for some $i\in\F_4$. We define the $\cA_{G,\lambda}$-module $M_{g,h}=\ku\{w_h,w_g\}$ by $\ku w_g\simeq \ku_g$ as $\cA_{G,\lambda}$-modules, $w_h\in M[h]$ and $x_j w_h=\delta_{j,i}w_g$ for all $j\in\F_4$.

\begin{prop}\label{le: extensions of one dimensional simple mod}
The extensions of one-dimensional $\cA_{G,\lambda}$-modules are either trivial or isomorphic to $M_{g,h}$ for some $g,h\in\ker\chi_z$. Hence $\cA_{G,\lambda}$ is of wild representation type.
\end{prop}

\begin{proof}
Let $M$ be an extension of $\ku_h$ by $\ku_g$. Then $M=M[g]\oplus M[h]$ as $\ku^G$-modules and $M[g]\simeq\ku_g$ as $\cA_{G,\lambda}$-modules. Since $x_i\cdot M[h]\subset M[g_i^{-1}h]$, the first part follows.

For the second part we can easily see that $\Ext_{\cA_{G,\lambda}}^1(\ku_g,\ku_h)$ is either $1$ or $0$ for all $g,h\in\ker\chi_z$. Then the separated quiver of $\cA_{G,\lambda}$ is wild. The details for this proof are similar to \cite[Proposition 26]{AV2}.
\end{proof}

\subsection{Is $\cA_{G,\lambda}$ spherical?}\label{subsec:spherical}

\quad A Hopf algebra $H$ is \emph{spherical} \cite{BaW-adv} if there is $\omega\in G(H)$ such that
\begin{align}\label{eq:omega-square-antipode}
\Ss^2(x) &= \omega x \omega^{-1}\quad\forall\, x\in H\,\mbox{ and}\\
\label{eq:omega-traza}
\tr_V(\omega) &= \tr_V(\omega^{-1})\quad\forall\, V\in\Irr H\quad\mbox{by \cite[Proposition 2.1]{AAGTV}.}
\end{align}

\begin{prop}\label{prop:agl spherical}
$\BV(V)\#\ku^G$ is spherical iff $\chi_G^2=1$. Moreover, $(\cA_{G,\lambda},\chi_G)$ with $\lambda\neq0$ is spherical iff $(\chi_{G|\ker\chi_z})^2=1$.
\end{prop}

\begin{proof}
It is a straightforward computation to see that $\chi_G$ satisfies \eqref{eq:omega-square-antipode} using \eqref{eq:TV smash ku a la G}. Let $V\in\Irr\cA_{G,\lambda}$. If $\dim V=12$, then $V$ is projective and therefore $\tr_{V}(\chi_G^{\pm1})=0$ \cite[Proposition 6.10]{BaW-tams}. If $V=\ku_h$ with $h\in\ker\chi_z$, then \eqref{eq:omega-traza} holds iff $\chi_G(h)=\pm1$.
\end{proof}

\begin{example}
Let $(\cdot,g,\chi_G)$ be the faithful principal YD-realization in Example \ref{ex:para spherical}. Then $(\cA_{G,\lambda},\chi_G)$ is a spherical Hopf algebra with non involutory pivot.
\end{example}

Any spherical Hopf algebra $H$ has an associated tensor category $\underline{\Rep}(H)$ which is a quotient of $\Rep(H)$,
see \cite{AAnMGV,BaW-adv,BaW-tams} for the background of this subject. Moreover, $\underline{\Rep}(H)$ is semisimple but rarely is a fusion category in the sense of \cite{ENO}, {\it i. ~e.} $\underline{\Rep}(H)$ rarely has a finite number of irreducibles. One hopes to find new examples of fusion categories as tensor subcategories of $\underline{\Rep}(H)$ for a suitable $H$. However, this is not possible for $H=\cA_{G,\lambda}$, see below.

\begin{Rem}\label{ex:spherical no new}
Assume that $(\cA_{G,\lambda},\chi_G)$ is spherical. Then only the one-dimensi-onal simple modules survive in $\underline{\Rep}(\cA_{G,\lambda})$ since the other simple modules are projective. Then $\underline{\Rep}(\cA_{G,\lambda})$ is equivalent to $\underline{\Rep}(\BV(V)\#\ku^{\ker\chi_z})$ by Subsection \ref{subsec:decomposition}, where the pivot $\chi_{G|\ker\chi_z}$ is involutory. Hence any fusion subcategory of $\underline{\Rep}(\cA_{G,\lambda})$ is equivalent to $\Rep(K)$, with $K$ a semisimple quasi-Hopf algebra, by \cite[Proposition 2.12]{AAGTV}.
\end{Rem}

\section*{Appendix}

\quad The next tables  describe the structure of the $12$-dimensional simple modules of $\cA_{G,\lambda}$. These were used in Lemma \ref{le:dim12}.
\begin{table}[!hbp]
\centering
\caption{Action of the generators $x_i$ on $L_1^g=\cA_{G,\lambda} \be_1^g$}\label{table:module L1g}
\begin{tabular}{l|c|c|c|c}
\hline\hline
Linear basis of $L_1^g$&$x_0\cdot$&$x_1\cdot$&$x_{\omega}\cdot$&$x_{\omega^2}\cdot$\Tstrut\Bstrut\\
\hline
$c_{1}=x_0x_1x_0x_{\omega}x_0x_1x_{\omega}x_0x_{\omega^2}\delta_g$&$0$&$0$&$-f(g)c_6$&$-f(g)c_{10}$\Tstrut\\
$c_2=x_0x_1x_0x_{\omega}x_0x_{\omega^2}\delta_g=-f(g)\be_1^g$&$0$&$0$&$-c_5$&$-c_9$\\
$c_{3}=x_0x_1x_{\omega}x_0x_1x_{\omega}x_0x_{\omega^2}\delta_g$&$0$&$c_1$&$f(g)c_{12}$&$0$\\
$c_4=x_0x_1x_{\omega}x_0x_{\omega^2}\delta_g$&$0$&$c_2$&$c_{11}$&$0$\\
$c_5=x_0x_{\omega}x_0x_1x_{\omega}x_0x_{\omega^2}\delta_g$&$0$&$c_7$&$0$&$-c_3$\\
$c_6=x_0x_{\omega}x_0x_{\omega^2}\delta_g$&$0$&$c_{8}$&$0$&$-c_4$\\
$c_{7}=x_1x_0x_{\omega}x_0x_1x_{\omega}x_0x_{\omega^2}\delta_g$&$c_1$&$0$&$0$&$-f(g)c_{12}$\\
$c_{8}=x_1x_0x_{\omega}x_0x_{\omega^2}\delta_g$&$c_2$&$0$&$0$&$c_{11}$\\
$c_9=x_1x_{\omega}x_0x_1x_{\omega}x_0x_{\omega^2}\delta_g$&$c_3$&$0$&$-c_7$&$0$\\
$c_{10}=x_1x_{\omega}x_0x_{\omega^2}\delta_g$&$c_4$&$0$&$-c_{8}$&$0$\\
$c_{11}=x_{\omega}x_0x_1x_{\omega}x_0x_{\omega^2}\delta_g$&$c_5$&$c_9$&$0$&$0$\\
$c_{12}=x_{\omega}x_0x_{\omega^2}\delta_g$&$c_6$&$c_{10}$&$0$&$0$
\\
\hline
\end{tabular}
\end{table}

\clearpage

\begin{table}[!htp]
\centering
\caption{Action of the generators $x_i$ on $L_2^g=\cA_{G,\lambda} \be_2^g$}\label{table:module L2g}
\begin{tabular}{l|c|c|c|c}
\hline\hline
Linear basis of $L_2^g$&$x_0\cdot$&$x_1\cdot$&$x_{\omega}\cdot$&$x_{\omega^2}\cdot$\Tstrut\Bstrut\\
\hline
$c_{1}=x_0x_1x_0x_{\omega}x_0x_1x_{\omega}x_{\omega^2}\delta_g$&$0$&$0$&$c_{6}$&$-f(g)c_{10}$\Tstrut\\
$c_2=x_0x_1x_0x_{\omega}x_{\omega^2}\delta_g$&$0$&$0$&$-c_{5}$&$-c_9$\\
$c_{3}=x_0x_1x_{\omega}x_0x_1x_{\omega}x_{\omega^2}\delta_g$&$0$&$c_1$&$-c_{12}$&$0$\\
$c_4=x_0x_1x_{\omega}x_{\omega^2}\delta_g$&$0$&$c_2$&$c_{11}$&$0$\\
$c_5=x_0x_{\omega}x_0x_1x_{\omega}x_{\omega^2}\delta_g=f(g)\be_2^g$&$0$&$c_7$&$0$&$-c_3$\\
$c_6=x_0x_1x_0x_{\omega}x_0x_1x_{\omega}x_0x_{\omega^2}\delta_g$&$0$&$-f(g)c_{8}$&$0$&$f(g)c_4$\\
$\quad\quad-x_0x_{\omega}x_{\omega^2}\delta_g$&&&&\\
$c_7=x_1x_0x_{\omega}x_0x_1x_{\omega}x_{\omega^2}\delta_g$&$c_1$&$0$&$0$&$-c_{12}$\\
$c_8=x_1x_0x_{\omega}x_{\omega^2}\delta_g$&$c_2$&$0$&$0$&$c_{11}$\\
$c_{9}=x_1x_{\omega}x_0x_1x_{\omega}x_{\omega^2}\delta_g$&$c_3$&$0$&$-c_7$&$0$\\
$c_{10}=x_1x_{\omega}x_{\omega^2}\delta_g$&$c_4$&$0$&$-c_8$&$0$\\
$c_{11}=x_{\omega}x_0x_1x_{\omega}x_{\omega^2}\delta_g$&$c_5$&$c_9$&$0$&$0$\\
$c_{12}=x_1x_0x_{\omega}x_0x_1x_{\omega}x_0x_{\omega^2}\delta_g-x_{\omega}x_{\omega^2}\delta_g$&$c_6$&$-f(g)c_{10}$&$0$&$0$
\\
\hline
\end{tabular}
\end{table}

\begin{table}[!hbp]
\caption{Action of the generators $x_i$ on $L_3^g=\cA_{G,\lambda} \be_3^g$}\label{table:module L3g}
\centering
\begin{tabular}{l|c|c|c|c}
\hline\hline
Linear basis of $L_3^g$&$x_0\cdot$&$x_1\cdot$&$x_{\omega}\cdot$&$x_{\omega^2}\cdot$\Tstrut\Bstrut\\
\hline
$c_{1}=x_0x_1x_0x_{\omega}x_0x_1x_{\omega^2}\delta_g$&$0$&$0$&$c_{6}$&$-c_{10}$\Tstrut\\
$c_2=x_0x_1x_0x_{\omega^2}\delta_g$&$0$&$0$&$-c_{5}$&$-c_9$\\
$c_{3}=x_0x_1x_{\omega}x_0x_1x_{\omega^2}\delta_g$&$0$&$c_1$&$c_{12}$&$0$\\
$c_4=x_0x_1x_{\omega^2}\delta_g$&$0$&$c_2$&$c_{11}$&$0$\\
$c_5=x_0x_{\omega}x_0x_1x_{\omega^2}\delta_g$&$0$&$c_7$&$0$&$-c_3$\\
$c_6=x_0x_1x_{\omega}x_0x_1x_{\omega}x_0x_{\omega^2}\delta_g$&$0$&$c_{8}$&$0$&$f(g)c_4$\\
$\quad\quad-f(g)x_0x_{\omega^2}\delta_g$&&&&\\
$c_7=x_1x_0x_{\omega}x_0x_1x_{\omega^2}\delta_g=f(g)\be_3^g$&$c_1$&$0$&$0$&$c_{12}$\\
$c_8=x_0x_1x_0x_{\omega}x_0x_1x_{\omega}x_0x_{\omega^2}\delta_g$&$-f(g)c_2$&$0$&$0$&$-f(g)c_{11}$\\
$\quad\quad-f(g)x_1x_0x_{\omega^2}\delta_g$&&&&\\
$c_{9}=x_1x_{\omega}x_0x_1x_{\omega^2}\delta_g$&$c_3$&$0$&$-c_7$&$0$\\
$c_{10}=x_0x_1x_0x_{\omega}x_0x_1x_{\omega}x_{\omega^2}\delta_g$&$-f(g)c_4$&$0$&$c_8$&$0$\\
$\quad\quad-f(g)x_1x_{\omega^2}\delta_g$&&&&\\
$c_{11}=x_{\omega}x_0x_1x_{\omega^2}\delta_g$&$c_5$&$c_9$&$0$&$0$\\
$c_{12}=x_1x_{\omega}x_0x_1x_{\omega}x_0x_{\omega^2}\delta_g$&$-c_6$&$-c_{10}$&$0$&$0$\\
$\quad\quad+x_0x_1x_{\omega}x_0x_1x_{\omega}x_{\omega^2}\delta_g-f(g)x_{\omega^2}\delta_g$&&&&
\\
\hline
\end{tabular}
\end{table}

\clearpage

\begin{table}[!htp]
\centering
\caption{Action of the generators $x_i$ on $L_4^g=\cA_{G,\lambda} \be_4^g$}\label{table:module L4g}
\begin{tabular}{l|c|c|c|c}
\hline\hline
Linear basis of $L_4^g$&$x_0\cdot$&$x_1\cdot$&$x_{\omega}\cdot$&$x_{\omega^2}\cdot$\Tstrut\Bstrut\\
\hline
$c_{1}=x_0x_1x_0x_{\omega}x_0\delta_g$&$0$&$0$&$-c_{6}$&$-c_{10}$\Tstrut\\
$c_2=x_0x_1x_0x_{\omega}x_0x_1x_{\omega}x_0\delta_g$&$0$&$0$&$-f(g)c_{5}$&$-c_9$\\
$c_{3}=x_0x_1x_{\omega}x_0\delta_g-x_0x_1x_0x_{\omega^2}\delta_g$&$0$&$c_1$&$c_{12}$&$0$\\
$c_4=x_0x_1x_{\omega}x_0x_1x_{\omega}x_0\delta_g-x_0x_1x_0x_{\omega}x_0x_1x_{\omega^2}\delta_g$&$0$&$c_2$&$c_{11}$&$0$\\
$c_5=x_0x_{\omega}x_0\delta_g$&$0$&$c_7$&$0$&$-c_3$\\
$c_6=x_0x_{\omega}x_0x_1x_{\omega}x_0\delta_g$&$0$&$c_{8}$&$0$&$-c_4$\\
$c_7=x_1x_0x_{\omega}x_0\delta_g$&$c_1$&$0$&$0$&$-c_{12}$\\
$c_8=x_1x_0x_{\omega}x_0x_1x_{\omega}x_0\delta_g$&$c_2$&$0$&$0$&$-c_{11}$\\
$c_{9}=x_1x_{\omega}x_0\delta_g-x_1x_0x_{\omega^2}\delta_g$&$c_3$&$0$&$-c_7$&$0$\\
$c_{10}=x_1x_{\omega}x_0x_1x_{\omega}x_0\delta_g-x_1x_0x_{\omega}x_0x_1x_{\omega^2}\delta_g$&$c_4$&$0$&$-c_8$&$0$\\
$\quad\quad=f(g)\be_4^g$&&&&\\
$c_{11}=x_0x_1x_{\omega}x_0x_1x_{\omega}x_0x_{\omega^2}\delta_g-f(g)x_0x_{\omega^2}\delta_g$&$c_5$&$c_9$&$0$&$0$\\
$\quad\quad+f(g)x_{\omega}x_0\delta_g$&&&&\\
$c_{12}=-x_0x_{\omega}x_0x_1x_{\omega^2}\delta_g+x_{\omega}x_0x_1x_{\omega}x_0\delta_g$&$c_6$&$c_{10}$&$0$&$0$
\\
\hline
\end{tabular}
\end{table}

\begin{table}[!hbp]
\centering
\caption{Action of the generators $x_i$ on $L_5^g=\cA_{G,\lambda} \be_5^g$}\label{table:module L5g}
\begin{tabular}{l|c|c|c|c}
\hline\hline
Linear basis of $L_5^g$&$x_0\cdot$&$x_1\cdot$&$x_{\omega}\cdot$&$x_{\omega^2}\cdot$\Tstrut\Bstrut\\
\hline
$c_{1}=x_0x_1x_0x_{\omega}\delta_g$&$0$&$0$&$-c_{6}$&$c_{10}$\Tstrut\\
$c_2=x_0x_1x_0x_{\omega}x_0x_1x_{\omega}\delta_g$&$0$&$0$&$-c_{5}$&$c_9$\\
$c_{3}=x_0x_1x_0x_{\omega}x_0x_1x_{\omega}x_0x_{\omega^2}\delta_g$&$0$&$f(g)c_1$&$-f(g)c_{12}$&$0$\\
$\quad\quad+f(g)x_0x_1x_{\omega}\delta_g$&&&&\\
$c_4=x_0x_1x_{\omega}x_0x_1x_{\omega}\delta_g-x_0x_1x_0x_{\omega}x_0x_{\omega^2}\delta_g$&$0$&$c_2$&$c_{11}$&$0$\\
$\quad\quad=f(g)\be_5^g$&&&&\\
$c_5=x_0x_1x_0x_{\omega}x_0x_1x_{\omega}x_0\delta_g+f(g)x_0x_{\omega}\delta_g$&$0$&$f(g)c_7$&$0$&$c_3$\\
$c_6=x_0x_{\omega}x_0x_1x_{\omega}\delta_g-f(g)x_0x_{\omega^2}\delta_g$&$0$&$c_{8}$&$0$&$c_4$\\
$c_7=x_1x_0x_{\omega}\delta_g$&$c_1$&$0$&$0$&$c_{12}$\\
$c_8=x_1x_0x_{\omega}x_0x_1x_{\omega}\delta_g$&$c_2$&$0$&$0$&$c_{11}$\\
$c_{9}=x_1x_0x_{\omega}x_0x_1x_{\omega}x_0x_{\omega^2}\delta_g+f(g)x_1x_{\omega}\delta_g$&$c_3$&$0$&$-f(g)c_7$&$0$\\
$c_{10}=x_1x_{\omega}x_0x_1x_{\omega}\delta_g-x_1x_0x_{\omega}x_0x_{\omega^2}\delta_g$&$c_4$&$0$&$-c_8$&$0$\\
$c_{11}=x_0x_{\omega}x_0x_1x_{\omega}x_0x_{\omega^2}\delta_g$&$c_5$&$c_9$&$0$&$0$\\
$\quad\quad+x_1x_0x_{\omega}x_0x_1x_{\omega}x_0\delta_g+f(g)x_{\omega}\delta_g$&&&&\\
$c_{12}=x_{\omega}x_0x_1x_{\omega}\delta_g-x_0x_{\omega}x_0x_{\omega^2}\delta_g$&$c_6$&$c_{10}$&$0$&$0$
\\
\hline
\end{tabular}
\end{table}

\clearpage

\begin{table}[!htp]
\caption{Action of the generators $x_i$ on $L_6^g=\cA_{G,\lambda} \be_6^g$}\label{table:module L6g}
\centering
\begin{tabular}{l|c|c|c|c}
\hline\hline
Linear basis of $L_6^g$&$x_0\cdot$&$x_1\cdot$&$x_{\omega}\cdot$&$x_{\omega^2}\cdot$\Tstrut\Bstrut\\
\hline
$c_{1}=x_0x_1x_0\delta_g$&$0$&$0$&$-c_{6}$&$-c_{10}$\Tstrut\\
$c_2=x_0x_1x_0x_{\omega^2}x_0x_1\delta_g$&$0$&$0$&$-c_{5}$&$c_9$\\
$c_{3}=x_0x_1x_0x_{\omega}x_0x_1x_{\omega}x_{\omega^2}\delta_g+f(g)x_0x_1\delta_g$&$0$&$f(g)c_1$&$c_{12}$&$0$\\
$c_4=x_0x_1x_{\omega}x_0x_1\delta_g-x_0x_1x_0x_{\omega}x_{\omega^2}\delta_g$&$0$&$c_2$&$c_{11}$&$0$\\
$c_5=-x_0x_1x_{\omega}x_0x_1x_{\omega}x_0\delta_g+f(g)x_0\delta_g$&$0$&$c_7$&$0$&$c_3$\\
$c_6=x_0x_{\omega}x_0x_1\delta_g$&$0$&$c_{8}$&$0$&$-c_4$\\
$c_7=-x_0x_1x_0x_{\omega}x_0x_1x_{\omega}x_0\delta_g+f(g)x_1x_0\delta_g$&$f(g)c_1$&$0$&$0$&$c_{12}$\\
$c_8=x_1x_0x_{\omega}x_0x_1\delta_g$&$c_2$&$0$&$0$&$c_{11}$\\
$c_{9}=x_1x_0x_{\omega}x_0x_1x_{\omega}x_0\delta_g$&$c_3$&$0$&$-c_7$&$0$\\
$\quad\quad-x_0x_1x_0x_{\omega}x_0x_1x_{\omega}\delta_g+f(g)x_1\delta_g$&&&&\\
$c_{10}=x_1x_{\omega}x_0x_1\delta_g-x_1x_0x_{\omega}x_{\omega^2}\delta_g$&$c_4$&$0$&$-c_8$&$0$\\
$c_{11}=x_0x_{\omega}x_0x_1x_{\omega}x_{\omega^2}\delta_g-x_1x_{\omega}x_0x_1x_{\omega}x_0\delta_g$&$c_5$&$c_9$&$0$&$0$\\
$\quad\quad-x_0x_1x_{\omega}x_0x_1x_{\omega}\delta_g+f(g)\delta_g=f(g)\be_6^g$&&&&\\
$c_{12}=-x_0x_1x_0x_{\omega}x_0x_1x_{\omega}x_0x_{\omega^2}\delta_g$&$f(g)c_6$&$c_{10}$&$0$&$0$\\
$\quad\quad+f(g)x_{\omega}x_0x_1\delta_g-f(g)x_0x_{\omega}x_{\omega^2}\delta_g$&&&&
\\
\hline
\end{tabular}
\end{table}

\begin{table}[!hbp]
\caption{Weight of the vectors $c_i$ in the case $G'=\F_4\rtimes C_6$}\label{table:pesos}
\centering
\begin{tabular}{r|c|c|c|c|c|c}
\hline\hline
     &$L_1^g$&$L_2^g$&$L_3^g$&$L_4^g$&$L_5^g$&$L_6^g$\Tstrut\Bstrut\\
\hline
$c_{1}$&$(0,t^3)g$&$(\omega,t^4)g$&$(0,t^5)g$&$(\omega^2,t)g$&$(\omega^2,t^2)g$&$(\omega,t^3)g$\Tstrut\\
 $c_2$ &$g$&$(\omega,t)g$&$(0,t^2)g$&$(\omega^2,t^4)g$&$(\omega^2,t^5)g$&$(\omega,1)g$\\
$c_{3}$&$(1,t^4)g$&$(\omega,t^5)g$&$(1,1)g$&$(0,t^2)g$&$(0,t^3)g$&$(\omega,t^4)g$\\
  $c_4$&$(1,t)g$&$(\omega,t^2)g$&$(1,t^3)g$&$(0,t^5)g$&$g$&$(\omega,t)g$\\
$c_5$&$(1,t^5)g$&$g$&$(1,t)g$&$(\omega^2,t^3)g$&$(\omega^2,t^4)g$&$(0,t^5)g$\\
$c_6$&$(1,t^2)g$&$(0,t^3)g$&$(1,t^4)g$&$(\omega^2,1)g$&$(\omega^2,t)g$&$(0,t^2)g$\\
$c_{7}$&$(0,t^4)g$&$(\omega^2,t^5)g$&$g$&$(1,t^2)g$&$(1,t^3)g$&$(\omega^2,t^4)g$\\
$c_{8}$&$(0,t)g$&$(\omega^2,t^2)g$&$(0,t^3)g$&$(1,t^5)g$&$(1,1)g$&$(\omega^2,t)g$\\
$c_9$&$(\omega,t^5)g$&$(\omega^2,1)g$&$(\omega,t)g$&$(0,t^3)g$&$(0,t^4)g$&$(\omega^2,t^5)g$\\
$c_{10}$&$(\omega,t^2)g$&$(\omega^2,t^3)g$&$(\omega,t^4)g$&$g$&$(0,t)g$&$(\omega^2,t^2)g$\\
$c_{11}$&$(\omega,1)g$&$(0,t)g$&$(\omega,t^2)g$&$(1,t^4)g$&$(1,t^5)g$&$g$\\
$c_{12}$&$(\omega,t^3)g$&$(0,t^4)g$&$(\omega,t^5)g$&$(1,t)g$&$(1,t^2)g$&$(0,t^3)g$
\\
\hline
\end{tabular}
\end{table}

\end{document}